\documentclass[11pt]{amsart}

\usepackage{amssymb, amsmath, , color, hyperref, url, amsxtra,  fullpage}
\usepackage{float}
\usepackage{placeins}
\usepackage{pdfpages}
\usepackage{graphicx}
\usepackage{mathtools}

\DeclarePairedDelimiter{\floor}{\lfloor}{\rfloor}
\usepackage{tikz}
\usepackage{geometry}
\newcommand{\boxWithLabel}[4]
{
	\draw (#2,#3) -- (#2-#1,#3) -- (#2-#1,#3-#1) -- (#2,#3-#1) -- (#2,#3);
	\node (a) at (#2-0.5*#1,#3-0.5*#1) {#4};
}

\allowdisplaybreaks

\newtheorem{theorem}{Theorem}[section]
\newtheorem{corollary}[theorem]{Corollary}
\newtheorem{lemma}[theorem]{Lemma}

\theoremstyle{definition}
\theoremstyle{conjecture}

\begin{document}

\title{Using Periodicity to Obtain Partition Congruences}

\author{ Ali H. Al-Saedi}
\address{Ali H. Al-Saedi}
\address{Oregon State University, Corvallis, OR 97331, USA}
\email{alsaedia@math.oregonstate.edu}

\begin{abstract}
In this paper, we generalize recent work of  Mizuhara, Sellers, and Swisher that gives a method for establishing  restricted plane partition congruences based on a bounded number of calculations. Using periodicity for partition functions, our extended technique could be a useful tool to prove congruences for certain types of combinatorial functions based on a bounded number of calculations. As applications of our result, we establish  new and existing restricted plane partition congruences, restricted plane overpartition congruences and several examples of restricted partition congruences. 
\end{abstract}



\maketitle


\section{\textbf{Introduction and Statements of Results}}
\subsection{Partitions and plane partitions}
A {\it{partition}} of a positive integer $n$ is a non-increasing sequence of positive integers that sum to $n$. The total number of partitions of $n$ is denoted by $p(n)$. We can define $p(n)$ on the set of all integers by setting $p(0)=1$ and $p(n)=0$ for all $n<0$. One can also consider partitions where the parts are restricted to a specific set $S$ of integers. For example, let $S$ be the set of positive integers, then $p(n;S)$ denotes the number of partitions of $n$ into parts from S. Clearly $p(n)=p(n;\mathbb{N}).$ 

For example, the partitions of $n=5$ are
$$5, 4+1, 3+2, 3+1+1, 2+2+1, 2+1+1+1, 1+1+1+1+1.$$
Thus $p(5)=7$ and if $S$ is the set of odd integers, then $p(5;S)=3.$

We write $|\lambda|=|(\lambda_{1},\lambda_{2},\cdots,\lambda_{k})| = n$  to indicate that $\lambda=(\lambda_{1},\lambda_{2},\cdots,\lambda_{k})$ is a partition of $n$. A {\it{Ferrers-Young}} diagram of a partition $\lambda$ of $n$ is a left-justified rectangular array of $n$ boxes, or cells, with a row of length $\lambda_{j}$ for each part $\lambda_{j}$ of $\lambda$. For example, the Ferrers-Young diagram of $\lambda=(6,4,3,1)$ is as follows.
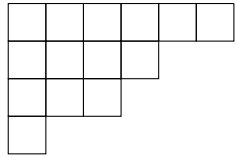
\begin{equation}\nonumber
\begin{tikzpicture}[scale=1]
	
		\boxWithLabel{0.5}{0}{4}{}
		\boxWithLabel{0.5}{0.5}{4}{}
		\boxWithLabel{0.5}{1}{4}{}
		\boxWithLabel{0.5}{1.5}{4}{}
		\boxWithLabel{0.5}{2}{4}{}
		\boxWithLabel{0.5}{2.5}{4}{}
		
		\boxWithLabel{0.5}{0}{3.5}{}
		\boxWithLabel{0.5}{0.5}{3.5}{}
		\boxWithLabel{0.5}{1}{3.5}{}
		\boxWithLabel{0.5}{1.5}{3.5}{}
		
		\boxWithLabel{0.5}{0}{3}{}
		\boxWithLabel{0.5}{0.5}{3}{}
		\boxWithLabel{0.5}{1}{3}{}
		
		\boxWithLabel{0.5}{0}{2.5}{}
				
	\end{tikzpicture}
\end{equation}

Ramanujan's beautiful partition congruences \cite{ramanujan1921congruence}, which state that for all $n\geq 0$,
\begin{align*}
p(5n+4)  & \equiv 0 \pmod{5} \\
p(7n+5)  & \equiv 0 \pmod{7} \\
p(11n+6) & \equiv 0 \pmod{11}
\end{align*}
have inspired a vast number of mathematicians to study and investigate special arithmetic properties of partitions, as well as interesting restricted partition functions and generalizations such as overpartitions and plane partitions. For example see \cite{andrews},\cite{andrewsandbruce},\cite{andrews1988dyson},\cite{atkin1954},\cite{bruce},\cite{bringmann2010dyson},\cite{corteel2004overpartitions},\cite{garvan1988}, \cite{lovejoy2008rank}, \cite{ono2000distribution} and \cite{rushforth} to mention a few. 

The generating function for $p(n)$ is due to Euler and is given by

\begin{equation}\label{eq1}
\sum_{n=0}^{\infty} p(n)q^{n}=\prod_{n=1}^{\infty} \frac{1}{1-q^n}.
\end{equation}

We can also consider partitions where parts are from a multiset $S$ such that each repeated number is treated independently.

For example, consider the multiset
$$S=\{1_{1},1_{2},2_{1},2_{2},2_{3},3\},$$ where repeated numbers have different indices. Then $p(2;S)=6$ since the partitions of $2$ with parts from $S$ are
$$2_{1},2_{2},2_{3},1_{1}+1_{1},1_{1}+1_{2},1_{2}+1_{2}.$$ 
Note that the order in the multiset gives an implied order to the repeated numbers.

Each partition can be considered as a one dimensional array of parts, and MacMahon \cite{andrews}  extended this idea to a two-dimensional array. A {\it{plane partition}} $\lambda$ of a positive integer $n$ is a two-dimensional array of positive integers $n_{i,j}$ that sum to $n$, such that the array is the Ferrers diagram of a partition, and the entries are non-increasing from left to right and also from top to bottom. Letting $i$ denote the row and $j$ the column of $n_{i,j},$ this means that for all $i, j\geq 0$,
\begin{align*}
n_{i,j}\geq n_{i+1,j},\\
n_{i,j}\geq n_{i,j+1}.
\end{align*}
Correspondingly, the entries $n_{i,j}$ are called the parts of $\lambda,$ and the number of plane partitions of $n$ is denoted by $pl(n).$ For example, the plane partitions for $n=3$ are as follows. \\
\begin{equation}\nonumber
\begin{tikzpicture}[scale=1]
	
		\boxWithLabel{0.5}{0}{4}{3}
				
		\boxWithLabel{0.5}{1}{4}{2}
		\boxWithLabel{0.5}{1.5}{4}{1}

		\boxWithLabel{0.5}{2.5}{4}{1}
		\boxWithLabel{0.5}{3}{4}{1}
		\boxWithLabel{0.5}{3.5}{4}{1}

		\boxWithLabel{0.5}{1}{3}{2}
		\boxWithLabel{0.5}{1}{2.5}{1}

		\boxWithLabel{0.5}{2.5}{3}{1}
		\boxWithLabel{0.5}{3}{3}{1}
		\boxWithLabel{0.5}{2.5}{2.5}{1}
		
		\boxWithLabel{0.5}{2.5}{1.5}{1}
		\boxWithLabel{0.5}{2.5}{1}{1}
		\boxWithLabel{0.5}{2.5}{0.5}{1}
		
	\end{tikzpicture}
\end{equation}
Thus, $pl(3)=6.$

MacMahon's challenge was to establish a nice generating function for $pl(n)$. However, it was not easy, it took him nearly twenty years \cite{andrews},\cite{macmahon2001combinatory} to prove that

\begin{equation}\label{eq2}
PL(q)=\sum_{n=0}^{\infty}{pl(n)q^{n}}=\prod_{n=1}^{\infty} \frac{1}{(1-q^n)^{n}}.
\end{equation}
He also considered a restricted form of plane partition that must have at most $r$ rows and $c$ columns. The generating function is given by 

\begin{equation}\label{eq3}
PL_{r,c}(q)=\sum_{n=0}^{\infty} pl_{r,c}(n)q^{n}=\prod_{i=1}^{r}\prod_{j=1}^{c} \frac{1}{1-q^{i+j-1}},
\end{equation}
where $pl_{r,c}(n)$ denotes the number of plane partitions of $n$ with at most $r$
rows and $c$ columns. By fixing $r$ and letting $c\longrightarrow \infty$, we obtain the generating function for {\it{r-rowed plane partitions}}, which are plane partitions with at most $r$ rows. The generating function is given by

\begin{equation}\label{eq4}
PL_{r}(q)=\sum_{n=0}^{\infty} pl_{r}(n)q^{n}=\prod_{i=1}^{r}\prod_{j=1}^{\infty} \frac{1}{1-q^{i+j-1}}=\prod_{n=1}^{\infty} \frac{1}{(1-q^n)^{min\{r,n\}}},
\end{equation}
where $pl_{r}(n)$ denotes the number of $r$-rowed plane partitions of $n$.

\subsection{Periodicity and plane partition congruences}
The goal of this paper is to generalize a result of Mizuhara, Sellers, and Swisher \cite{periodic} which uses periodicity to study plane partition congruences. Kwong and others have done extensive studies on the periodicity of certain rational functions, including partition generating functions, for example see  \cite{kwong3}, \cite{kwong1}, \cite{kwong2}, \cite{newman}, and \cite{nijenhuis1987}. Before we state the result in \cite{periodic}, we need a couple of definitions.

Let $$A(q)=\sum_{n=0}^{\infty}{a_{n}q^{n}} \in \mathbb{Z}[[q]]$$ be a formal power series with integer coefficients, and let $d, \ell$ be positive integers. We say $A(q)$ is {\it{purely periodic}} with period $d$ modulo $\ell$ if, for all $n\geq 0$,
$$a_{n+d}\equiv a_{n}\pmod{\ell}.$$
The smallest such period for $A(q)$, denoted $\pi_{\ell}(A)$, is called the {\it{minimal period of}} $A(q)$ modulo $\ell$. In this paper, periodic  means purely periodic. 

Mizuhara, Sellers, and  Swisher \cite{periodic} consider the class of plane partition congruences of the form      
\begin{equation}\label{eqth1}
\sum_{i=1}^{s}{pl_{\ell}(n\ell+a_{i})}\equiv \sum_{j=1}^{t}{pl_{\ell}(n \ell +b_{i})}\pmod{\ell},\;\;\;
\mbox{for all}\;\; n\geq 0.
\end{equation}

\begin{theorem}[Mizuhara, Sellers, Swisher \cite{periodic}]{} \label{th1}
Fix positive integers $s, t$ and nonnegative integers $0\leq a_{i},b_{j}\leq \ell-1$ for each $1 \leq i \leq s, 1 \leq j \leq t$. For a prime $\ell$, if
\begin{equation}\nonumber
\sum_{i=1}^{s}{pl_{\ell}(n\ell+a_{i})}\equiv \sum_{j=1}^{t}{pl_{\ell}(n \ell +b_{i})}\pmod{\ell}
\end{equation}
holds for all $n <\pi_{\ell}(F_{\ell})/\ell$, then it holds for all $n \geq 0,$ where 
$$F_{\ell}(q):=\prod_{n=1}^{\ell-1}\frac{1}{(1-q^n)^n}. $$
\end{theorem}

Theorem \ref{th1} states that for a prime $\ell$, one can look only at the finite set of $\ell$-rowed plane partition numbers $pl_{\ell}(\ell n+a_{i})$ and their finite sums for $0\leq n<\pi_{\ell}(F_{\ell})/\ell$ to see if there is a congruence of the form (\ref{eqth1}) that holds for all $n$. For example, taking $\ell=2$, if there is a congruence, then it is one of the following possible choices  
$$pl_{2}(2n)\equiv 0\pmod{2}$$
$$pl_{2}(2n+1)\equiv 0\pmod{2}$$
$$pl_{2}(2n)\equiv pl_{2}(2n+1)\pmod{2}.$$
If any of the congruences above holds for $0\leq n<\pi_{2}(F_{2})/2$, then it holds for all $n\geq 0$. The reason this technique works is because $F_{\ell}(q)$ is periodic, which is due to a theorem of Kwong \cite{kwong1} (see Theorem \ref{kwong}).

Theorem \ref{th1} was used to prove several plane partition congruences, some previously known to Gandhi \cite{gandhi} and others previously unknown.  

\begin{theorem}[\cite{periodic}]{} \label{th2}
The following hold for all $n \geq 0$,
\begin{align}\label{g1}
pl_{2}(2n + 1) \equiv pl_{2}(2n) \pmod{2}
\end{align}
\begin{align}\label{g2}
pl_{3}(3n + 2) \equiv 0 \pmod{3}
\end{align}
\begin{align}\label{g3}
pl_{3}(3n + 1) \equiv pl_{3}(3n) \pmod{3}
\end{align}
\begin{align}\label{g4}
pl_{5}(5n + 2) \equiv pl_{5}(5n+4) \pmod{5}
\end{align}
\begin{align}\label{g5}
pl_{5}(5n + 1) \equiv pl_{3}(5n+3) \pmod{5}
\end{align}
\begin{align}\label{g6}
pl_{7}(7n + 2)+ pl_{7}(7n+3)\equiv pl_{7}(7n + 4)+ pl_{7}(7n+5) \pmod{7}.
\end{align}

\end{theorem}
 
Note that the identities \eqref{g1}, \eqref{g3},\eqref{g4} and  \eqref{g5} were previously shown by Gandhi \cite{gandhi}, while \eqref{g2} and \eqref{g6} are proved in \cite{periodic}. 

We are now able to state the main result of this paper. We generalize Theorem \ref{th1} to a wider class of $q$-series, and to include prime power moduli. 

\begin{theorem}\label{thA}
	Fix a prime $\ell$, and let $ N, K, \delta$ be any positive integers. Let $A(q), B(q) \in \mathbb{Z}[[q]]$ such that
	$A(q):=\sum_{n=0}^{\infty}{\alpha(n)q^{n}}$
	is periodic modulo $\ell^N$ with minimal period $\pi_{\ell^N}(A)=\delta K$ and suppose that $B(q):=\sum_{m=0}^{\infty}{\beta(m)q^{m}}$, where $\beta(0)\equiv 1 \pmod{\ell^N}$ and $\beta(m)\equiv 0 \pmod{\ell^N}$ for $m\not\equiv 0 \pmod{\delta }.$ 
	Define
	$$G(q):=A(q)\cdot  B(q):=\sum_{k=0}^{\infty}\lambda(k) q^{k}.$$
	Fix  positive integers $s, t$ and nonnegative integers $0\leq a_{i},b_{j}\leq \delta-1$ for each $1\leq i\leq s, 1\leq j\leq t$. If 
	\begin{align*}
	\sum_{i=1}^{s}\lambda(\delta  n+a_{i})\equiv\sum_{j=1}^{t} \lambda(\delta  n+b_{j})\pmod{\ell^N},
	\end{align*}
	holds for all $0\leq n < \pi_{\ell^N}(A)/\delta$, then it holds for all $n\geq 0$. 
\end{theorem}
The generality of Theorem \ref{thA} gives potential for many more applications, which we discuss further in Section 2. Two such examples for plane partitions are as follows. We prove in Theorem \ref{power} that for all $n\geq 0$,
\begin{align*}
pl_{8}(8n)+pl_{8}(8n+1) \equiv pl_{8}(8n+3) \pmod{2},
\end{align*}
\begin{align*}
pl_{9}(9n+1)\equiv pl_{9}(9n+8)\pmod{3}.
\end{align*}
 
The rest of this paper is organized as follows. In Section 2, we review some preliminaries, including a useful theorem of Kwong \cite{kwong1} and further definitions of types of partition functions for which we can apply Theorem \ref{thA}. In Section 3, we prove several congruences for a variety of partition congruences as applications of our main theorem. In section 4, we prove Theorem \ref{thA}. In Section 5, we conclude with final remarks.
\section{Preliminaries}

Before we state a result of Kwong \cite{kwong1}, we recall some necessary definitions from \cite{periodic}. 
For an integer $n$ and prime $\ell$, define $ord_{\ell}(n)$ to be the unique nonnegative integer such that $$\ell^{ord_{\ell}(n)}\cdot m=n,$$ where $m$ is an integer and $\ell\nmid m$. In addition, we call $m$ the {\it{$\ell$-free part}} of $n$.
 
For a finite multiset of positive integers $S$, we define $m_{\ell}(S)$ to be the $\ell$-free part of $lcm\{n|n\in S\}$, and $b_{\ell}(S)$ to be the least nonnegative integer such that 
$$\ell^{b_{\ell}(S)}\geq \sum_{n\in S}\ell^{ord_{\ell}(n)}.$$

\begin{theorem}[Kwong,\cite{kwong1}]\label{kwong}
Fix a prime $\ell$, and a finite multiset $S$ of positive integers. Then for any positive integer $N$, $$A(q)=\sum_{n=0}^{\infty} {p(n;S)q^{n}}$$ is periodic modulo $\ell^{N}$, with minimal period $$\pi_{\ell^{N}}(A)=\ell^{N+b_{\ell}(S)-1}\cdot m_{\ell}(S).$$
\end{theorem}

We note that Theorem \ref{kwong} can be applied to calculate the minimum periodicity modulo prime powers of any rational functions of the form
\begin{equation}
R_{k}(e_{1},e_{2},\ldots, e_{k};q):=\frac{1}{(1-q)^{e_{1}}(1-q^{2})^{e_{2}}\cdots (1-q^{k})^{e_{k}}},
\end{equation}
where $k$ is a positive integer and $e_{i}$ are nonnegative integers for $1\leq i\leq k$. For the positive integers $e_{i}$, consider the multiset of positive integers $i_{j}$ associated with $e_{i}$ for $1\leq j\leq e_{i}$, that is
$$S_{k,\overline{e}}=\{i_{j}|   1\leq i \leq k, e_{i}\geq1,  1\leq j\leq e_{i}\},$$
where $\overline{e}:=(e_{1},e_{2},\ldots, e_{k}).$
Then by the standard partition theory argument, we get
$$\sum_{n\geq 0}{p(n;S_{k,\overline{e}})q^{n}}=R_{k}(e_{1},e_{2},\ldots, e_{k};q).$$
\begin{lemma}\label{lemmaR}
Fix a prime $\ell$ and a nonnegative integer $N$, then $R_{k}(e_{1},e_{2},\ldots, e_{k};q)$ is a periodic q-series modulo $\ell^{N}$ with minimal period
$$\pi_{\ell^{N}}(R_{k})=\ell^{N+b_{\ell}(S_{k,\overline{e}})-1}\cdot m_{\ell}(S_{k,\overline{e}}).$$
\end{lemma}
\begin{proof}
By the previous argument,
$$\sum_{n\geq 0}{p(n;S_{k,\overline{e}})q^{n}}=R_{k}(e_{1},e_{2},\ldots, e_{k};q).$$
The rest follows by Theorem \ref{kwong}.
\end{proof}

For example, let $k=4$, and $\overline{e}=(1,0,2,3)$. Thus,
$$R_{4}(1,0,2,3;q)=\frac{1}{(1-q)(1-q^3)^2(1-q^4)^3}$$
is generated by the multiset 
$$S_{4,\overline{e}}=\{1,3_{1},3_{2},4_{1},4_{2},4_{3}\}.$$
In particular, for $\ell=3$ and $N=1$, we calculate $b_{3}(S_{4,\overline{e}})$ and $m_{3}(S_{4,\overline{e}})$ where
$$3^{b_{3}(S_{4,\overline{e}})}\geq 3^0+2\cdot 3^1+ 3\cdot 3^0=10 \implies b_{3}(S_{4,\overline{e}})=3, $$
and 
$$m_{3}(S_{4,\overline{e}})=4.$$
Thus by Lemma \ref{lemmaR}, the minimal periodicity modulo $3$ of  $R_{4}(1,0,2,3;q)$ is given by
$$\pi_{3}(R_{4})=3^3\cdot 4.$$

Let $\ell$ be a prime, consider the special case of Lemma \ref{lemmaR}  with $k=\ell-1$ and $e_{i}=i$ for $1\leq i\leq \ell-1$. Then  
$$F_{\ell}(q)=R_{\ell-1}(1,2,\ldots, \ell-1;q)=\prod_{n=1}^{\ell-1}{\frac{1}{(1-q^{n})^{n}}}.$$ 
We then have the following immediate corollary which is a particular case of Corollary 2.4 in \cite{periodic}.
\begin{corollary}\label{corR}
For a prime $\ell$, and a positive integer $N$, $F_{\ell}(q)$ is periodic modulo $\ell^N$ with minimal period 
$$\pi_{\ell^{N}}(F_{\ell})=\ell^{N+b_{\ell}(S_{\ell-1,\overline{e}})-1}\cdot m_{\ell}(S_{\ell-1,\overline{e}}),$$
where $\overline{e}=(1,\dots, \ell-1)$.
\end{corollary}

We also state the following elementary lemmas which can be easily proved inductively on $N$.
\begin{lemma}\label{lemmaL}
For any prime $\ell$ and positive integers $j$ and $N$,
\begin{align*}
(1-q^j)^{\ell^N}\equiv (1-q^{j\ell^N})\pmod{\ell}.
\end{align*}
\end{lemma}

\begin{lemma}\label{lemmaL2}
For any prime $\ell$ and positive integers $j$ and $N$,
\begin{align*}
(1-q^j)^{\ell^N} \equiv (1-q^{j\ell})^{\ell^{N-1}} \pmod{\ell^N}.
\end{align*}

\end{lemma}

We observe in the following lemma that the restricted plane partition generating functions are always of the shape needed in Theorem \ref{thA}.
\begin{lemma}\label{lemma2}
For a prime $\ell$ and a positive integer $N$, then 
$$PL_{\ell^N}(q)\equiv F_{\ell^N}(q)\cdot\sum_{m\geq 0}{\beta(m)q^{\ell^N m}} \pmod{\ell},$$ 
$$PL_{\ell^N}(q)\equiv F_{\ell^N}(q)\cdot\sum_{m\geq 0}{\beta^\prime (m)q^{\ell m}} \pmod{\ell^N},$$ 
where $\beta(m), \beta^\prime(m)\in \mathbb{N}.$
\end{lemma}
\begin{proof}
We recall the generating function of $\ell^N$-rowed plane partitions from \eqref{eq4}
$$PL_{\ell^N}(q)=F_{\ell^N}(q)\cdot \prod_{n=\ell^N}^{\infty} \frac{1}{(1-q^n)^{\ell^N}}.$$
By Lemma \eqref{lemmaL} and Lemma\eqref{lemmaL2}, one can easily see that
$$\prod_{n=\ell^N}^{\infty} \frac{1}{(1-q^n)^{\ell^N}}\equiv \prod_{n=\ell^N}^{\infty} \frac{1}{(1-q^{n\ell^N})} \pmod{\ell}=:\sum_{m\geq 0}{\beta(m)q^{\ell^N m}} \pmod{\ell},$$

$$\prod_{n=\ell^N}^{\infty} \frac{1}{(1-q^n)^{\ell^N}}\equiv \prod_{n=\ell^N}^{\infty} \frac{1}{(1-q^{n\ell})^{\ell^{N-1}}} \pmod{\ell^N}=:\sum_{m\geq 0}{\beta^\prime(m)q^{\ell m}} \pmod{\ell^N},$$
where $\beta(m), \beta^\prime(m)\in \mathbb{N}$. Therefore,
$$PL_{\ell}(q)\equiv F_{\ell^N}(q)\cdot \prod_{n=\ell^N}^{\infty} \frac{1}{(1-q^{n\ell^N})} \pmod{\ell}\equiv F_{\ell^N}(q)\cdot\sum_{m\geq 0}{\beta(m)q^{\ell^N m}} \pmod{\ell},$$
and 
$$PL_{\ell}(q)\equiv F_{\ell^N}(q)\cdot \prod_{n=\ell^N}^{\infty} \frac{1}{(1-q^{n\ell})^{\ell^{N-1}}} \pmod{\ell^N}\equiv F_{\ell^N}(q)\cdot\sum_{m\geq 0}{\beta^\prime(m)q^{\ell m}} \pmod{\ell^N}.$$
\end{proof}

\section{Applications of Theorem \ref{thA}}
\subsection{Plane partition congruences involving prime powers}
In \cite{gandhi} and \cite{periodic}, elementary combinatorial methods were used to prove some plane partition congruences modulo primes and prime powers. With less effort and a different technique, we apply Theorem \ref{thA} to reprove some of these congruences and establish new equivalences.

\begin{theorem} \label{power}
The following hold for all $n\geq 0,$
\begin{align}\label{p1}
pl_{4}(4n+3)\equiv 0\pmod{2}
\end{align}
\begin{align}\label{p2}
pl_{4}(4n)\equiv pl_{4}(4n+1)\equiv pl_{4}(4n+2)\pmod{2}
\end{align}
\begin{align}\label{p3}
pl_{8}(8n)+pl_{8}(8n+1) \equiv pl_{8}(8n+3)\pmod{2}
\end{align}
\begin{align}\label{p4}
pl_{8}(8n+5)\equiv pl_{8}(8n+6) \equiv pl_{8}(8n+7)\equiv 0 \pmod{2}
\end{align}
\begin{align}\label{p5}
pl_{9}(9n+1)\equiv pl_{9}(9n+8)\pmod{3}.
\end{align}
\end{theorem}

We note that (\ref{p1}) and (\ref{p2}) are shown by Gandhi \cite{gandhi}, and \eqref{p4} is previously reported in \cite{periodic}, while (\ref{p3}) and (\ref{p5}) are new to the literature.

\begin{proof}
We note that
\begin{align*}
\sum_{n=0}^{\infty}{pl_{4}(n)q^n}&=\frac{1}{(1-q)(1-q^2)^2 (1-q^3)^3}\cdot \prod_{n=4}^{\infty}{\frac{1}{(1-q^n)^4}}\\
&=\left(\frac{1}{(1-q)(1-q^3)^3}\right)\cdot \left(\frac{1}{(1-q^2)^2}\cdot\prod_{n=4}^{\infty}{\frac{1}{(1-q^n)^4}}\right).
\end{align*}
By Lemma \ref{lemmaR},
\begin{align*}
A(q)&=\sum_{n=0}^{\infty}\alpha(n) q^n:=\frac{1}{(1-q)(1-q^3)^3}\\
    &=1+q+q^2+4q^3+4q^4+4q^5+10q^6+10q^7+10q^8+20q^9+20q^{10}+20q^{11}+35q^{12}+\cdots
\end{align*}
is periodic modulo 2 with minimal period $\pi_{2}(A)=12$. Also, we use Lemma \ref{lemmaL} to observe that
\begin{align*}
B(q)&=\sum_{n=0}^{\infty}\beta(n) q^n:=\frac{1}{(1-q^2)^2}\cdot \prod_{n=4}^{\infty}\frac{1}{(1-q^n)^4}\\
    &\equiv \frac{1}{(1-q^4)}\cdot\prod_{n=4}^{\infty}\frac{1}{(1-q^{4n})} \pmod{2}.
\end{align*}
Thus $\beta(0)=1$ and $\beta(n)\equiv 0\pmod{2}$ for all $n\not\equiv 0\pmod{4}$, and hence the series $B(q)$ and its coefficients satisfy the desired congruences conditions of Theorem \ref{thA}.

We see that the congruences 
$$pl_{4}(4n+3)\equiv 0\pmod{2},$$
$$pl_{4}(4n)\equiv pl_{4}(4n+1)\equiv pl_{4}(4n+2)\pmod{2}$$
hold for $n=0,1$ and $2$.  For $\ell=2, N=1,\delta=4$, we apply Theorem \ref{thA} and hence the equivalences (\ref{p1}) and (\ref{p2}) hold for all $n\geq 0$.

To prove the congruences (\ref{p3}) and (\ref{p4}), again we use Lemma \ref{lemmaL} to observe that
\begin{align*}
&\sum_{n=0}^{\infty}{pl_{8}(n)q^n}=\frac{1}{(1-q)(1-q^2)^2 (1-q^3)^3(1-q^4)^4(1-q^5)^5(1-q^6)^6(1-q^7)^7}\cdot \prod_{n=8}^{\infty}{\frac{1}{(1-q^n)^8}}\\
& \equiv \frac{1}{(1-q)(1-q^2)^2 (1-q^3)^3(1-q^5)^5(1-q^6)^6(1-q^7)^7}\cdot \left( \frac{1}{(1-q^{16})} \cdot \prod_{n=8}^{\infty}{\frac{1}{(1-q^{8n})}}\right) \pmod{2}.
\end{align*}
By Lemma \ref{lemmaR}, the quotient 
$$\sum_{n=0}^{\infty}{\alpha(n)q^n}:=\frac{1}{(1-q)(1-q^2)^2 (1-q^3)^3(1-q^5)^5(1-q^6)^6(1-q^7)^7}$$ 
is periodic modulo 2 with minimal period $2^5 \cdot 105$. Maple programming shows that the congruences
$$pl_{8}(8n)+pl_{8}(8n+1)\equiv pl_{8}(8n+3)\pmod{2},$$
$$pl_{8}(8n+5)\equiv pl_{8}(8n+6) \equiv pl_{8}(8n+7)\equiv 0 \pmod{2}$$
hold for all $0\leq n\leq \frac{2^5 \cdot 105}{2^3}.$  Thus for $\ell=2,N=1,\delta=8$,  Theorem \ref{thA} confirms the congruences (\ref{p3}) and (\ref{p4}) hold for all $n\geq 0$. 

To prove (\ref{p5}), we use the same method to see that
\begin{align*}
&\sum_{n=0}^{\infty}{pl_{9}(n)q^n}=\frac{1}{(1-q)(1-q^2)^2 (1-q^3)^3(1-q^4)^4(1-q^5)^5(1-q^6)^6(1-q^7)^7(1-q^8)^8}\cdot \prod_{n=9}^{\infty}{\frac{1}{(1-q^n)^9}}\\
& \equiv \frac{1}{(1-q)(1-q^2)^2(1-q^4)^4 (1-q^5)^5(1-q^6)^6(1-q^7)^7(1-q^8)^8}\cdot \left(\frac{1}{(1-q^9)}\cdot \prod_{n=9}^{\infty}{\frac{1}{(1-q^{9n})}}\right) \pmod{3}.
\end{align*}
Again, by Lemma \ref{lemmaR}, the quotient
$$\sum_{n=0}^{\infty}{\alpha(n)q^n}:=\frac{1}{(1-q)(1-q^2)^2(1-q^4)^4(1-q^5)^5(1-q^6)^6(1-q^7)^7(1-q^8)^8}$$
is periodic modulo 3 with minimal period $3^4\cdot 280$. Again by maple programming, we confirm that for all $0\leq n< \frac{3^4\cdot 280}{3^2},$
\begin{align*}
pl_{9}(9n+1)\equiv pl_{9}(9n+8)\pmod{3},
\end{align*}
as desired.
\end{proof}

\subsection{Overpartitions and plane overpartitions}
An {\it{overpartition}} of a positive integer $n$ is a partition of $n$ in which the first occurrence (equivalently, the last occurrence) of a part may be overlined. We denote the number of overpartition of $n$ by $\overline{p}(n)$. Since the overlined parts form a partition into distinct parts and the non-overlined parts form an ordinary partition, the generating function of overpartitions is given by
\begin{equation}\label{eq1}
\sum_{n=0}^{\infty}\overline{p}(n)q^{n}=\prod_{n=1}^{\infty} \frac{1+q^n}{1-q^n}= 1+2q+4q^{2}+8q^{3}+14q^{4}+ \cdots .
\end{equation}

A {\it{plane overpartition}} is a plane partition where  in each row the last occurrence of an integer can be overlined or not and all the other occurrences of this integer are not overlined, and furthermore in each column the first occurrence of an integer can be overlined or not and all the other occurrences of this integer are overlined. The total number of plane overpartitions of $n$ is denoted by $\overline{pl}(n)$ and the generating function of plane overpatitions \cite{vuletic2007} is given by
\begin{equation}\label{eq1.3}
\overline{PL}(q)=\sum_{n=0}^{\infty}\overline{pl}(n)q^n=\prod_{n=1}^{\infty} \frac{(1+q^n)^{n}}{(1-q^n)^{n}} .
\end{equation}

For example, the plane overpartitions for $n=3$ are as follows.
$$
{
\begin{tikzpicture}[scale=1]
	
		\boxWithLabel{0.5}{0}{4}{3}
		\boxWithLabel{0.5}{1}{4}{$\bar{3}$}
				
		\boxWithLabel{0.5}{2}{4}{2}
		\boxWithLabel{0.5}{2.5}{4}{1}

		\boxWithLabel{0.5}{3.5}{4}{$\bar{2}$}
		\boxWithLabel{0.5}{4}{4}{1}

		\boxWithLabel{0.5}{5}{4}{2}
		\boxWithLabel{0.5}{5.5}{4}{$\bar{1}$}
		
		\boxWithLabel{0.5}{6.5}{4}{$\bar{2}$}
		\boxWithLabel{0.5}{7}{4}{$\bar{1}$}
		
		\boxWithLabel{0.5}{8}{4}{1}
		\boxWithLabel{0.5}{8.5}{4}{1}
		\boxWithLabel{0.5}{9}{4}{1}
		
		\boxWithLabel{0.5}{10}{4}{1}
		\boxWithLabel{0.5}{10.5}{4}{1}
		\boxWithLabel{0.5}{11}{4}{$\bar{1}$}

		\boxWithLabel{0.5}{2}{3}{2}
		\boxWithLabel{0.5}{2}{2.5}{1}
		
		\boxWithLabel{0.5}{3.5}{3}{$\bar{2}$}
		\boxWithLabel{0.5}{3.5}{2.5}{1}
		
		\boxWithLabel{0.5}{5}{3}{2}
		\boxWithLabel{0.5}{5}{2.5}{$\bar{1}$}
		
		\boxWithLabel{0.5}{6.5}{3}{$\bar{2}$}
		\boxWithLabel{0.5}{6.5}{2.5}{$\bar{1}$}

		\boxWithLabel{0.5}{8}{3}{1}
	    \boxWithLabel{0.5}{8.5}{3}{1}
		\boxWithLabel{0.5}{8}{2.5}{$\bar{1}$}
		
		\boxWithLabel{0.5}{10}{3}{1}
	    \boxWithLabel{0.5}{10.5}{3}{$\bar{1}$}
		\boxWithLabel{0.5}{10}{2.5}{$\bar{1}$}

		\boxWithLabel{0.5}{8}{1.5}{1}
	    \boxWithLabel{0.5}{8}{1}{$\bar{1}$}
		\boxWithLabel{0.5}{8}{0.5}{$\bar{1}$}

		\boxWithLabel{0.5}{10}{1.5}{$\bar{1}$}
	    \boxWithLabel{0.5}{10}{1}{$\bar{1}$}
		\boxWithLabel{0.5}{10}{0.5}{$\bar{1}$}
		
	\end{tikzpicture}
	}
$$

Thus, $\overline{pl}(3)=16.$ 

A {\it{$k$-rowed plane overpartition}} of an integer $n$ is a plane overpartition with at most $k$ rows. The total number of the $k$-rowed plane overpartitions of $n$ is denoted by $\overline{pl}_{k}(n)$.
To get the generating function of $k$-rowed plane overpartitions, we recall the following theorem.

\begin{theorem}[\cite{vuletic2007}]{}\label{thm5}
The generating function for plane overpartitions which fit in an $k \times n$ box is $$\prod_{i=1}^{k}{\prod_{j=1}^{n}{\frac{1+q^{i+j-1}}{1-q^{i+j-1}}}}.$$
\end{theorem}

\begin{lemma}\label{lemma3}
For a fixed positive integer $k$, the generating function for $k$-rowed plane overpartitions is given by
\begin{equation}\label{eq2}
\overline{PL}_{k}(q):=\sum_{n=0}^{\infty}{\overline{pl}_{k}(n)q^{n}}
=\prod_{n=1}^{\infty}{\frac{(1+q^{n})^{\small \mbox{min}\{k,n\}}}{(1-q^{n})^{\small \mbox{min}\{k,n\}}}}.
\end{equation}
\end{lemma}
\begin{proof}
By fixing $k$ and letting $n\rightarrow\infty$ in Theorem \ref{thm5}, we get
$$\prod_{i=1}^{k}{\prod_{j=1}^{\infty}{\frac{1+q^{i+j-1}}{1-q^{i+j-1}}}}=\frac{(1+q)(1+q^{2})^{2}\cdots (1-q^{k-1})^{k-1}}{(1-q)(1-q^{2})^{2}\cdots (1-q^{k-1})^{k-1}}\cdot \prod_{n=k}^{\infty} \frac{(1+q^{n})^{k}}{(1-q^{n})^{k}}=\prod_{n=1}^{\infty}{\frac{(1+q^{n})^{\small \mbox{min}\{k,n\}}}{(1-q^{n})^{\small \mbox{min}\{k,n\}}}}.$$
\end{proof}

Also we note that  for a prime $\ell$ and a positive integer $N$, the restricted plane overpartition generating function $\overline{PL}_{\ell^N}(q)$ is of the form $A(q)\cdot B(q)$ where $A(q)$ and $B(q)$ are described in Theorem \ref{thA}.
\begin{lemma}
For a prime $\ell$ and a positive integer $N$, then $$\overline{PL}_{\ell^N}(q)\equiv R_{k}(m_{1},\dots,m_{k};q)\cdot\sum_{m\geq 0}{\beta(m)q^{\ell^N m}} \pmod{\ell},$$ 
$$\overline{PL}_{\ell^N}(q)\equiv R_{k^\prime}(m_{1}^\prime,\dots,m_{k^\prime}^\prime;q)\cdot\sum_{m\geq 0}{\beta^\prime(m)q^{\ell m}} \pmod{\ell^N},$$ 
for some positive integers $k,k^\prime$ and nonnegative integers $ m_{1},\dots,m_{k},m_{1}^\prime,\dots,m_{k^\prime}^\prime, \beta(m),\beta^\prime(m).$
\end{lemma}
\begin{proof}
First, let $N=1$. By Lemma \ref{lemma3}, the generating function of $\ell$-rowed plane overpartitions is given by
\begin{align*}
\overline{PL}_{\ell}(q)&=\frac{(1+q)(1+q^{2})^{2}\cdots (1-q^{\ell-1})^{\ell-1}}{(1-q)(1-q^{2})^{2}\cdots (1-q^{\ell-1})^{\ell-1}}\cdot \prod_{n\geq \ell} \frac{(1+q^{n})^{\ell}}{(1-q^{n})^{\ell}}\\
&=\frac{(1+q)(1+q^{2})^{2}\cdots {(1-q^{\frac{\ell-1}{2}})^{\frac{\ell-1}{2}}}} {(1-q)(1-q^{2})^{2}\cdots (1-q^{(\ell-1)})^{\ell-1}}\cdot (1+q^{\frac{\ell-1}{2}+1})^{\frac{\ell-1}{2}+1}\cdots (1+q^{\ell-1})^{\ell-1}\cdot \prod_{n\geq \ell} \frac{(1+q^{n})^{\ell}}{(1-q^{n})^{\ell}}.
\end{align*}
By factorizing the denominator of the front quotient, we note that there are some nonnegative integers $m_{1},\dots, m_{\ell-1}$  so that
\begin{align*}
\frac{(1+q)(1+q^{2})^{2}\cdots {(1-q^{\frac{\ell-1}{2}})^{\frac{\ell-1}{2}}}} {(1-q)(1-q^{2})^{2}\cdots (1-q^{(\ell-1)})^{\ell-1}}=\frac{1}{(1-q)^{m_{1}}(1-q^{2})^{m_{2}}\cdots (1-q^{\ell-1})^{m_{\ell-1}}}.
\end{align*}
Furthermore,
\begin{align*}
&(1+q^{\frac{\ell-1}{2}+1})^{\frac{\ell-1}{2}+1}\cdots (1+q^{\ell-1})^{\ell-1}\cdot  \prod_{n\geq \ell} \frac{(1+q^{n})^{\ell}}{(1-q^{n})^{\ell}}\\
&=(1+q^{\frac{\ell-1}{2}+1})^{\frac{\ell-1}{2}+1}\cdots (1+q^{\ell-1})^{\ell-1}\cdot \frac{(1+q^{\ell})^{\ell}(1+q^{\ell+1})^{\ell} \cdots }{(1-q^{\ell})^{\ell}(1-q^{\ell+1})^{\ell}\cdots (1-q^{2(\ell-1)})^{\ell}\cdots (1-q^{2\ell})^{\ell}\cdots }\\
&= \left(\frac{(1+q^{\frac{\ell+1}{2}})^{\frac{\ell+1}{2}}}{(1-q^{\ell+1})^{\ell}}\cdots \frac{(1+q^{\ell-1})^{\ell-1}}{(1-q^{2(\ell-1)})^{\ell}} \right)
\cdot \left(\frac{(1+q^{\ell})^{\ell}(1-q^{\ell+1})^{\ell} \cdots}{(1-q^{\ell})^{\ell}(1-q^{\ell+2})^{\ell}\cdots(1-q^{2\ell})^{\ell}\cdots(1-q^{2(\ell+1)})^{\ell}\cdots}\right)\\
&=\frac{1}{(1-q^{t_{1}})^{r_{1}}\cdots (1-q^{t_{j}})^{r_{j}}}\cdot \prod_{i\geq 1}{\frac{1}{(1-q^{n_{i}})^{\ell}}},
\end{align*}
for some nonnegative integers $r_{i}$ and positive integers $t_{i}$ and $n_{i}$. Therefore,
\begin{align*}
\overline{PL}_{\ell}(q)&=\frac{1}{(1-q)^{m_{1}}(1-q^{2})^{m_{2}}\cdots (1-q^{\ell-1})^{m_{\ell-1}}(1-q^{t_{1}})^{r_{1}}\cdots (1-q^{t_{j}})^{r_{j}}}\cdot \prod_{i\geq 1}{\frac{1}{(1-q^{n_{i}})^{\ell}}}.
\end{align*}
We reindex $m_{i}$ and $r_{i}$ so that
\begin{align*}
\overline{PL}_{\ell}(q)&=\frac{1}{(1-q)^{m_{1}}(1-q^{2})^{m_{2}}\cdots (1-q^k)^{m_{k}}}\cdot \prod_{i\geq 1}{\frac{1}{(1-q^{n_{i}})^{\ell}}}\\
&=R_{k}(m_{1},\dots,m_{k};q)\cdot \prod_{i\geq 1}{\frac{1}{(1-q^{n_{i}})^{\ell}}}.
\end{align*}
We can repeat the same process for $N>1$ to obtain
\begin{align*}
\overline{PL}_{\ell^N}(q)&=R_{k}(m_{1},\dots,m_{k};q)\cdot \prod_{i\geq 1}{\frac{1}{(1-q^{n_{i}})^{\ell^N}}}.
\end{align*}
Using Lemma \ref{lemmaL} and Lemma \ref{lemmaL2}, the rest follows.
\end{proof}
In particular, for $\ell=2,3,5$, we have the following generating functions, 

\[\overline{PL}_{\ell}(q)=\begin{cases}
      \frac{1}{(1-q^{2})}\cdot \left( \frac{1}{(1-q)^{2}}\cdot \prod_{n=2}^{\infty}{ \frac{(1+q^{n})^{2}}{(1-q^{n+1})^{2}}}\right) & \mbox{if}\; \ell=2\\
       \frac{1}{(1-q)^{2}(1-q^{4})}\cdot \left( \frac{1}{(1-q^2)^{3}(1-q^3)^{3}}\cdot\prod_{n=3}^{\infty}{\frac{(1+q^n)^{3}}{(1-q^{n+2})^{3}}}\right)  & \mbox{if}\; \ell=3 \\
       \frac{1}{(1-q)^{2}(1-q^2)^{3}(1-q^3)(1-q^4)^{2}(1-q^{8})^2}\cdot \left( \frac{1}{(1-q^3)^{5}(1-q^4)^{5}(1-q^5)^5 (1-q^7)^5}\cdot \prod_{n=5}^{\infty}{\frac{(1+q^n)^{5}}{(1-q^{n+3})^5}}\right)  & \mbox{if}\; \ell=5 \\
   \end{cases}
\]

Note that
\begin{align*}
\sum_{n=0}^{\infty}{\overline{pl}_{k}(n)q^{n}}&
=\prod_{n=1}^{\infty}{\frac{(1+q^{n})^{\small \mbox{min}\{k,n\}}}{(1-q^{n})^{\small \mbox{min}\{k,n\}}}}= \prod_{n=1}^{\infty}{\left(1+\frac{2q^{n}}{1-q^n}\right)^{\small \mbox{min}\{k,n\}}}\\
&\equiv 1 \pmod{2}.
\end{align*}
Thus, 
$$\overline{pl}_{k}(n)\equiv 0\pmod{2},$$
for any $k,n\geq 1.$
\begin{theorem}\label{thOverPartition}
The following holds for all $n\geq 0$,
\begin{align}\label{ov1}	\overline{pl}_{4}(4n+1)+\overline{pl}_{4}(4n+2)+\overline{pl}_{4}(4n+3)\equiv 0\pmod{4}
\end{align}
\end{theorem}
\begin{proof} 
Observe that
\begin{align*}
\overline{PL}_{4}(q)&=\left(\frac{(1+q)(1+q^2)^2(1+q^3)^3}{(1-q)(1-q^2)^2(1-q^3)^3} \right)\cdot \left( \prod_{n=4}^{\infty}{\frac{(1+q^{n})^{4}}{(1-q^{n})^{4}}}\right)\\
&=\left(\frac{1}{(1-q)(1-q^3)^3}\cdot\frac{(1+q)}{(1-q^2)^2}\cdot \frac{(1+q^2)^2}{(1-q^4)^4}\cdot\frac{(1+q^3)^3}{(1-q^6)^4}\right)\cdot\\
&\;\;\;\; \left((1+q^4)^4(1+q^6)^4 \cdot\frac{(1+q^5)^4}{(1-q^{5})^4}\cdot \prod_{n=7}^{\infty}{\frac{(1+q^n)^4}{(1-q^{n})^{4}}}\right)\\
&=\left(\frac{1}{(1-q)^2(1-q^2)^3(1-q^3)^6(1-q^6)}\right) \cdot\\
&\;\;\;\;\; \left((1+q^6)^4  \cdot\frac{(1+q^4)^4 (1+q^5)^4}{(1-q^4)^2 (1-q^{5})^4}\cdot \prod_{n=7}^{\infty}{\frac{(1+q^n)^4}{(1-q^{n})^4}}\right).
\end{align*}
Note that for all $n\geq 1$,
\begin{align*}
\frac{(1+q^n)^4}{(1-q^n)^4}=\left(1+\frac{2q^n}{1-q^n}\right)^4 \equiv 1 \pmod{4}.
\end{align*}
Therefore,
\begin{align*}
\frac{(1+q^5)^4}{(1-q^{5})^4}\cdot \prod_{n=7}^{\infty}{\frac{(1+q^n)^4}{(1-q^{n})^{4}}}\equiv 1 \pmod{4}.
\end{align*}
Also, we note that,
\begin{align*}
(1+q^6)^4 &\equiv 1+2q^{12}+q^{24} \pmod{4},\\
\frac{(1+q^4)^4}{(1-q^4)^2} &\equiv (1+q^4)^2 \pmod{4}.
\end{align*}
So, we let
\begin{align*}
A(q)&=\sum_{n=0}^{\infty}{\alpha(n)q^n}:=\frac{1}{(1-q)^2(1-q^2)^3(1-q^3)^6(1-q^6)},\\
B(q)&=\sum_{n=0}^{\infty}{\beta(n)q^n}:=(1+q^6)^4 \cdot\frac{(1+q^4)^4(1+q^5)^4 }{(1-q^4)^2 (1-q^{5})^4}\cdot \prod_{n=7}^{\infty}{\frac{(1+q^n)^4}{(1-q^{n})^{4}}}\\
&\equiv (1+2q^{12}+q^{24})(1+q^4)^2\pmod{4}.
\end{align*}
Then by Lemma \ref{lemmaR},
$$A(q)=R_{6}(2,3,6,0,0,1)$$
is periodic modulo 4 with minimal period $\pi_{4}(A)=2^5 \cdot 3$. By a calculation in Maple, we note that
$$\overline{pl}_{4}(4n+1)+\overline{pl}_{4}(4n+2)+\overline{pl}_{4}(4n+3)\equiv 0\pmod{4},$$
for all $0\leq n\leq \pi_{4}(A)/4.$
Hence letting $\ell=N=2, \delta=4$ and applying Theorem \ref{thA}, the congruence (\ref{ov1}) holds for all $n\geq 0$.
\end{proof}

\subsection{Congruences of partitions with parts at most $m$}

Let $m$ and $n$ be  positive integers. Define $p(n,m)$ to be the number of partitions of $n$ into parts with at most $m$. The generating function of $p(n,m)$ is given by
$$Q(q,m):=\sum_{n=0}^{\infty}{p(n,m)q^{n}}=\frac{1}{(1-q)(1-q^2)\cdots (1-q^{m})}.$$
Clearly, $Q(q,m)$ is generated by the finite set $T_{m}=\{1,2,\cdots, m\}$.  In other words, 
$$Q(q,m)=\sum_{n=0}^{\infty}{p(n;T_{m})q^{n}}.$$
Using Lemma \ref{lemmaR} for a prime $\ell$, one can see that

$$\pi_{\ell}(Q(q,\ell-1))=\ell \cdot m_{\ell}(T_{\ell-1}),$$
$$\pi_{\ell}(Q(q,\ell))=\ell^2 \cdot m_{\ell}(T_{\ell}).$$

\begin{theorem}
The following holds for all $n\geq 0$,
\begin{align}\label{p32}
p(3n+1,2)+p(3n+2,2)\equiv 0\pmod{3}
\end{align}
\begin{align}\label{p51}
p(10n+6,4)+p(10n+7,4)+p(10n+8,4)\equiv 0\pmod{5}
\end{align}
\begin{align}\label{p52}
p(10n+2,4)+p(10n+3,4)+p(10n+4,4)\equiv 0\pmod{5}.
\end{align}
\end{theorem}
\begin{proof}
Since $\pi_{3}(Q(q,2))=6$ and hence for $n=0,1$, we see
\begin{align*}
p(1,2)+p(2,2)=3 \equiv 0\pmod{3},\\
p(4,2)+p(5,2)=6\equiv 0\pmod{3}.
\end{align*}
Therefore, (\ref{p32}) holds for all $n \geq 0$.

For $\ell=5$, $\pi_{5}((Q(q,4))=60$. By a calculation in Maple, we verify the coefficients of $Q(q,4)$ modulo 5 in the following tables.\\

\FloatBarrier
\begin{table}[H]\label{table1}
\caption{Restricted Partitions Modulo 5 For \eqref{p51}}
\centering 
\begin{tabular}{c c c c } 
\hline\hline 
$n$ & $p(10n+6,4)$ & $p(10n+7,4)$ & $p(10n+8,4)$\\ [0.5ex] 
\hline 

0  & 4 & 1 & 0  \\ 
1  & 4 & 2 & 4  \\
2  & 1 & 0 & 4  \\
3  & 3 & 1 & 1  \\
4  & 0 & 2 & 3  \\
5  & 0 & 0 & 0  \\[1ex] 
\hline 
\end{tabular}
\label{table:nonlin} 
\end{table}

\noindent From Table 1, we note the congruence
$$p(10n+6,4)+p(10n+7,4)+p(10n+8,4)\equiv 0\pmod{5}$$
holds for all  $n=0,1,2,3,4,5.$\\

Furthermore, 

\FloatBarrier
\begin{table}[H]
\caption{Restricted Partitions Modulo 5 For \eqref{p52} } 
\centering 
\begin{tabular}{c c c c } 
\hline\hline 
$n$ & $p(10n+2,4)$ & $p(10n+3,4)$ & $p(10n+4,4)$ \\ [0.5ex] 
\hline 
0  & 2 & 3 & 0 \\ 
1  & 4 & 4 & 2 \\
2  & 1 & 0 & 4 \\
3  & 1 & 3 & 1 \\
4  & 0 & 4 & 1\\
5  & 0 & 0 & 0\\ [1ex] 
\hline 
\end{tabular}
\label{table:nonlin} 
\end{table}
Thus, we have from Table 2 for all $n=0,1,2,3,4,5$, 
$$p(10n+2,4)+p(10n+3,4)+p(10n+4,4)\equiv 0\pmod{5}.$$
By applying Theorem \ref{thA} for $A(q)=Q(q,4), B(q)=1,\delta=10$ and $N=1$ we deduce that (\ref{p51}) and (\ref{p52}) hold for all $n\geq 0.$
\end{proof}

\section{Proof of The Main Theorem}
We now present the proof of Theorem \ref{thA}.

\noindent{{\it Proof of Theorem \ref{thA}}. 
	Let $\ell$ be a prime, and $N, K, \delta$ be any positive integers. Suppose that $A(q), B(q) \in \mathbb{Z}[[q]]$ such that
	$A(q):=\sum_{n=0}^{\infty}{\alpha(n)q^{n}}$
	is periodic modulo $\ell^N$ with minimal period $\pi_{\ell^N}(A)=\delta K$ and  $B(q):=\sum_{m=0}^{\infty}{\beta(m)q^{m}}$, where $\beta(0)\equiv 1 \pmod{\ell^N}$ and $\beta(m)\equiv 0 \pmod{\ell^N}$ for $m\not\equiv 0 \pmod{\delta}.$ 
	Let
	$$G(q):=A(q)\cdot  B(q):=\sum_{k=0}^{\infty}\lambda(k) q^{k}.$$
	Since $\beta(m)\equiv 0 \pmod{\ell^N}$ for $m\not\equiv 0 \pmod{\delta},$ then
	$$B(q)\equiv\sum_{m=0}^{\infty}\beta(m\delta)q^{m\delta} \pmod{\ell^N}.$$
	Let $\beta^\prime(m):=\beta(m\delta)$ for all $m\geq 0$. Thus
	\begin{align*}
	\sum_{k=0}^{\infty} \lambda(k) q^{k} &\equiv \left(\sum_{n=0}^{\infty}{\alpha(n) q^{n}}\right)\cdot \left( \sum_{m=0}^{\infty}{\beta^\prime(m)q^{m\delta}}\right)\\
	& =\sum_{k=0}^{\infty}{\left(\sum_{i=0}^{\floor{\frac{k}{\delta}}}{\alpha(k-i\delta) \beta^\prime(i)}\right) q^k}\pmod{\ell^N} .
	\end{align*}
	Therefore, for $k\geq 0,$
	
	\begin{equation}\label{eq5}
	\lambda(k) \equiv\sum_{i=0}^{\floor{\frac{k}{\delta}}}{\alpha(k-i\delta) \beta^\prime(i)}\pmod{\ell^N}.
	\end{equation}
	Hence letting $k=n\delta+j$ in (\ref{eq5}), for  $n\geq 0$ and $0\leq j\leq \delta -1,$ we obtain
	
	\begin{equation}\label{eq6}
	\lambda(n\delta +j)\equiv \sum_{r=0}^{n}{\alpha(r\delta +j)\beta^\prime(n-r)}\pmod{\ell^N}.
	\end{equation}
	
	Notice that by (\ref{eq6}), for any $n\geq 0$, the congruence  
	\begin{equation}\label{eq7}
	\sum_{i=1}^{s}{\lambda(n\delta +a_{i})}\equiv \sum_{j=1}^{t}{\lambda(n\delta+b_{j})}\pmod{\ell^N}
	\end{equation}
	is equivalent to
	
	\begin{equation*}
	\sum_{i=1}^{s}{\sum_{r=0}^{n}{\alpha(r\delta+a_{i})\beta^\prime(n-r)}}\equiv \sum_{j=1}^{t}{\sum_{r=0}^{n}{\alpha(r\delta+b_{j})\beta^\prime(n-r)}}\pmod{\ell^N},
	\end{equation*}
	or in particular to
	
	\begin{equation*}
	\sum_{r=0}^{n}{\beta^\prime(n-r)}\left(\sum_{i=1}^{s}{\alpha(r\delta+a_{i})}\right)\equiv \sum_{r=0}^{n}{\beta^\prime(n-r)}\left(\sum_{j=1}^{t}{\alpha(r\delta+b_{j})}\right)\pmod{\ell^N}.
	\end{equation*}
	
	To prove (\ref{eq7}) holds for all $n\geq 0$, it thus suffices to prove that the congruence 
	\begin{equation}\label{eq8}
	\sum_{i=1}^{s}{\alpha(n\delta+a_{i})}\equiv \sum_{j=1}^{t}{\alpha(n\delta+b_{j})}\pmod{\ell^N}
	\end{equation}
	holds for all $n\geq 0.$
	
	By the hypothesis, (\ref{eq7}) holds for all $0\leq n < \pi_{\ell^N}(A)/\delta$. Thus for $0 \leq n < \pi_{\ell^N}(A)/\delta$, we see that
	
	\begin{equation}\label{eqAli1}
	\sum_{r=0}^{n}{\beta^\prime(n-r)}\left(\sum_{i=1}^{s}{\alpha(r\delta+a_{i})}\right)\equiv \sum_{r=0}^{n}{\beta^\prime(n-r)}\left(\sum_{j=1}^{t}{\alpha(r\delta+b_{j})}\right)\pmod{\ell^N}.
	\end{equation}
	Letting $n=0$, (\ref{eqAli1}) implies that 
	$
	\beta^\prime(0)(\sum_{i=1}^{s}{\alpha(a_{i})})\equiv \beta^\prime(0)(\sum_{j=1}^{t}{\alpha(b_{j})})\pmod{\ell^N}.
	$
	Since $\beta^\prime(0)\equiv 1 \pmod{\ell^N}$, thus $\sum_{i=1}^{s}{\alpha(a_{i})}\equiv \sum_{j=1}^{t}{\alpha(b_{j})}\pmod{\ell^N}$. For $n\geq 1$,
	\begin{multline}\label{eqAli2}
	\sum_{i=1}^{s}{\alpha(n\delta+a_{i})}+\sum_{r=0}^{n-1}{\beta^\prime(n-r)}\left(\sum_{i=1}^{s}{\alpha(r\delta+a_{i})}\right)\equiv\\
	\sum_{j=1}^{t}{\alpha(n\delta+b_{j})}+\sum_{r=0}^{n-1}{\beta^\prime(n-r)}\left(\sum_{j=1}^{t}{\alpha(r\delta+b_{j})}\right)\pmod{\ell^N}.
	\end{multline}
	We see recursively from (\ref{eqAli2}) that for all $0\leq n < \pi_{\ell^N}(A)/\delta$,
	\begin{equation*}
	\sum_{i=1}^{s}{\alpha(n\delta+a_{i})}\equiv \sum_{j=1}^{t}{\alpha(n\delta+b_{j})}\pmod{\ell^N}.
	\end{equation*}
	
	To finish the proof, it suffices to prove that (\ref{eq8}) holds for all $n\geq \pi_{\ell^N}(A)/\delta.$
	By hypothesis, there is some $K \in \mathbb{N}$ such that $\pi_{\ell^N}(A)=K\delta$. Fix an arbitrary integer $n \geq \pi_{\ell^N}(A)/\delta=K$. By the Division Algorithm, we can write $n=xK+y$ where $0\leq y<K$. Thus for each $1 \leq i \leq s$, and $1\leq j\leq t$, we have
	$$n\delta+a_{i}=x\cdot\pi_{\ell^N}(A)+(y\delta+a_{i}),$$ 
	$$n\delta+b_{j}=x\cdot\pi_{\ell^N}(A)+(y\delta+b_{j}).$$
	From this we see that
	$$n\delta+a_{i}\equiv y\delta+a_{i}\pmod{\pi_{\ell^N}(A)},$$
	$$n\delta+b_{j}\equiv y\delta+b_{j}\pmod{\pi_{\ell^N}(A)}.$$
	Since $A(q)$ is periodic modulo $\ell^N$ with minimal period $\pi_{\ell^N}(A)$, then for each $1\leq i \leq s$, and $1\leq j\leq t,$
	$$\alpha(n\delta+a_{i})\equiv \alpha(y\delta+a_{i})\pmod{\ell^N},$$
	$$\alpha(n\delta+b_{j})\equiv \alpha(y\delta+b_{j})\pmod{\ell^N}.$$
	Since  $0\leq y <  K =\pi_{\ell^N}(A)/\delta$, we have by our hypotheses that
	
	$$\sum_{i=1}^{s}{\alpha(y\delta+a_{i})}\equiv \sum_{j=1}^{t} {\alpha(y\delta+b_{j})}\pmod{\ell^N}.$$
	
	Therefore,
	
	$$\sum_{i=1}^{s}{\alpha(n\delta+a_{i})}\equiv\sum_{i=1}^{s}{\alpha(y\delta+a_{i})}\equiv\sum_{j=1}^{t} {\alpha(y\delta+b_{j})}\equiv\sum_{j=1}^{t} {\alpha(n\delta+b_{j})}\pmod{\ell^N},$$
	as desired.
	\qed
	\\

\section{Conclusion}

We have generalized the method of Mizuhara, Sellers, and Swisher \cite{periodic} to give a way to determine various congruences based on a bounded number of calculations. We note that as applications of Theorem \ref{thA}, we obtain new plane  partition and plane overpartition congruences. However, the results are limited to computing capabilities since, at least in our cases, increasing the primes leads to more involved coefficient calculations. We hope that further investigations may prove plane partition and plane overpartition congruences modulo higher primes and prime powers. In addition, it would be interesting to find examples of congruences for other types of combinatorial functions which can be proved by Theorem \ref{thA}.

\section{Acknowledgements}
I sincerely thank my advisor Professor Holly Swisher for her guidance and encouragement in carrying out this paper. The completion of this work would not have been done without her assistance and participation.
\\
\\
\bibliographystyle{plain}
\bibliography{ref}

\end{document}